\documentclass[11pt,twoside,reqno]{amsart}
\usepackage{eurosym}
\usepackage[colorlinks=true,citecolor=blue]{hyperref}
\usepackage{mathptmx, amsmath, amssymb, amsfonts, amsthm, mathptmx, enumerate, color}
\usepackage{graphicx}
\usepackage{epstopdf}

\newtheorem{definition}{Definition}[section]

\newtheorem{remark}{Remark}[section]

\newtheorem{lemma}{Lemma}[section]
\newtheorem{theorem}{Theorem}[section]

\theoremstyle{definition}

\numberwithin{equation}{section}
\allowdisplaybreaks
\input{tcilatex}

\begin{document}
\title[Hilfer-Hadmard fractional implicit differential equation ]{Some
existence and stability results of Hilfer-Hadmard fractional implicit
differential equation in a weighted space}
\author[L.~A.~Palve$^{1\ast }$, M. S.~Abdo, S. K.~Panchal]{ Laxman.~A.~Palve$%
^{1\ast }$, Mohammed~S. Abdo$^{1,2}$, Satish K. Panchal $^{1}$}
\maketitle

\setcounter{page}{1}

\vspace*{-0.6cm}

\begin{center}
{\footnotesize $^{1}$Department of Mathematics, Dr.Babasaheb Ambedkar
Marathwada University, Aurangabad, (M.S),431001, India.\\[0pt]
}

{\footnotesize $^{2}$Department of Mathematics, Hodeidah University,
Al-Hodeidah,3114, Yemen.\\[0pt]
}

{\footnotesize $^{\ast }$Email: laxmanpalve11@gmail.com\\[0pt]
Email: msabdo1977@gmail.com\\[0pt]
Email:drpanchalskk@gmail.com\\[0pt]
}
\end{center}

\vskip 4mm {\footnotesize \noindent \textbf{Abstract.} }

{\footnotesize This paper studies a nonlinear fractional implicit
differential equation (FIDE) with boundary conditions involving a
Hilfer-Hadamard type fractional derivative. We establish the equivalence
between the Cauchy-type problem (FIDE) and its mixed type integral equation
through a variety of tools of some properties of fractional calculus and
weighted spaces of continuous functions. The existence and uniqueness of
solutions are obtained. Further, the Ulam-Hyers and Ulam-Hyers-Rassias
stability are discussed. The arguments in the analysis rely on Schaefer
fixed point theorem, Banach contraction principle and generalized Gronwall
inequality. At the end, an illustrative example will be introduced to
justify our results. }



\renewcommand{\thefootnote}{} \footnotetext{$^*$Corresponding author.
\par
E-mail addresses: laxmanpalve11@gmail.com ( Laxman.~A.~Palve).}

\section{Introduction}

Fractional order differentiation is the generalization of classical integer
order differentiation. Derivatives and integrals of fractional orders play a
significant role to describe the irregular behavior and anomalous flow of
dynamical systems in physics, viscoelasticity, biology, electrochemistry,
diffusion process, and chaotic systems, and they are increasingly used to
model problems in fluid dynamics, control theory, finance, unsteady
aerodynamics and aeroelastic phenomena, viscoelasticity, electrodynamics of
complex medium, theory of population dynamics, and other areas of
application. There are various definitions of fractional derivatives, among
these definitions, Riemann-Liouville (1832), Riemann (1849),
GrunwaldLetnikov (1867), Caputo (1997), Hilfer (2000, \cite{RH19}), as well
as Hadamard (1891, \cite{Ha}) which are the most used.

In recent years, fractional differential equations have become invaluable
instrument occurs in many areas of mathematics, physics, chemistry,
engineering, bio-engineering and so on. A systematic presentation of the
applications of fractional differential equations in physics and engineering
can be found in the book of Oldham and Spanier \cite{OS}, Sabatier et al. 
\cite{SA}, Hilfer \cite{RH19}. The fundamental results in the theory of
fractional differential equations are gaining much importance and attention.
For more details, see the monographs of Kilbas \cite{AHJ16}, Samko \cite%
{SAO21}, Podlubny \cite{IP18}. In consequence, there has also been a surge
in the study of the theory of fractional differential equations \cite%
{BT,DR,DF,LV} and references therein. Over the last thirty years, the
advancement of stability for the functional equations was studied by Ulam 
\cite{Ul15}, Hyers \cite{Hy4} and this type of stability called Ulam-Hyers
stability. Thereafter improvement of Ulam-Hyers stability provided by
Rassias \cite{Ra10} in 1978. For some recent results of stability analysis
by different types of fractional derivative operator, we refer the reader to
a series of papers \cite{BB,KMN8,SJ3,R12,JS33,34,Ru35,VKE37}. Recently,
remarkable attention has been given to the existence of solutions of initial
value problem for differential equations of fractional order with
Hilfer-Hadamard derivative, one can see, \cite{PAJ6,PAJ7,MN15,VKE,JYM26}.
Howover, the studies about fractional boundary value problem with
Hilfer-Hadamard derivative are just a few, among these works \cite%
{AP3,BB,AP2,SB36}. In year 2018, Abbas et al. \cite{AB2} investigated some
existence of weak solutions of Hilfer-Hadamard fractional differential
equation%
\begin{equation}
_{H}D_{1+}^{\alpha ,\beta }u(t)=f({t,u(t)}),\quad 0<\alpha <1,\quad 0\leq
\beta \leq 1,\text{ }t\in I=[1,T]  \label{p1a}
\end{equation}%
\begin{equation}
_{H}I_{1^{+}}^{1-\gamma }u(1^{+})=\phi ,\quad ,\gamma =\alpha +\beta
(1-\alpha ).  \label{p1b}
\end{equation}%
Using M\"{o}nch's fixed point theorem associated with the technique of
measure of weak noncompactness.

In the sam year, Vivek et al. \cite{VKE}, discussed the existence and
different types of Ulam stability results to Hilfer--Hadamard fractional
implicit differential equation with nonlocal condition%
\begin{equation}
_{H}D_{1+}^{\alpha ,\beta }u(t)=f({t,u(t),}_{H}D_{1+}^{\alpha ,\beta
}u(t)),\quad 0<\alpha <1,\quad 0\leq \beta \leq 1,\text{ }t\in I=[1,b]
\label{p2a}
\end{equation}%
\begin{equation}
_{H}I_{1^{+}}^{1-\gamma }u(1^{+})=\sum\limits_{k=0}^{m}c_{i}u(\tau
_{i}),\quad ,\tau _{i}\in \lbrack 1,b],\quad \alpha \leq \gamma =\alpha
+\beta (1-\alpha ).  \label{p2b}
\end{equation}%
Using Schaefer's fixed point theorem and generalized Gronwall inequality.

However, to the best of our knowledge, there is no work on boundary value
problems with Hilfer-Hadamard fractional derivatives in the literature. The
aim of the present work is to study a new class of boundary value problems
of Hilfer-Hadamard-type fractional implicit differential equations with
boundary conditions and develop the existence, uniqueness, and stability
analysis for the solutions of such problems. In precise terms, we consider
the fractional implicit differential equation with boundary condition of the
form 
\begin{equation}
_{H}D_{1+}^{\alpha ,\beta }u(t)=f({t,u(t),}_{H}D_{1+}^{\alpha ,\beta
}u(t)),\quad 0<\alpha <1,\quad 0\leq \beta \leq 1,\text{ }t\in I=[1,b],
\label{eq1}
\end{equation}%
\begin{equation}
_{H}I_{1+}^{1-\gamma }c_{1}u(1^{+})+c_{2}u(b^{-})=\phi ,\quad \alpha \leq
\gamma =\alpha +\beta (1-\alpha ),\qquad \qquad \qquad \qquad  \label{eq2}
\end{equation}%
where $_{H}D_{1+}^{\alpha ,\beta }$ is the Hilfer-Hadamard type fractional
derivative of order $\alpha $ and type $\beta $, $_{H}I_{1+}^{1-\gamma }$ is
the left-sided mixed Hadamard fraction integral of order $1-\gamma ,$ $%
c_{1},c_{2},\phi \in 
\mathbb{R}
$ with $c_{1}+c_{2}\neq 0,$ $c_{2}\neq 0,$ and $f:I\times E\times
E\longrightarrow E$ is a function satisfied some conditions that will state
later.

This paper is organized as follows: In section \ref{se2}, we recall some
preliminaries and fundamental concepts of Hilfer-Hadamard type. Section \ref%
{se3}, contains the main results and is divided into two parts. Part one
dealt with the equivalence between the problem FIDE (\ref{eq1})-(\ref{eq2})
and the mixed type integral equation. Part two, we investigate the existence
and uniqueness results to problem FIDE (\ref{eq1})-(\ref{eq2}). In Section %
\ref{se4}, we discuss different types of stability analysis via generalized
Gronwall inequality for Hadamard type. Section \ref{se5}, contains
supporting an example. The last section of this paper displays the
conclusions.

\section{Preliminaries\label{se2}}

In this section, we introduce some notations, lemmas, definitions and
weighted spaces which are important in developing some theories throughout
this paper. By $C[I,E]$ we denote the Banach space of continuous functions ${%
\varphi :I\rightarrow }E$, with the norm $\left\Vert \varphi \right\Vert
_{C}=\sup \{\left\vert \varphi (t)\right\vert ;t\in I\}$. Let $L^{1}[I,E]$
be the Banach space of Lebesgue integrable functions ${\varphi :I\rightarrow 
}E$, with the norm $\left\Vert \varphi \right\Vert
_{L^{1}}=\int_{I}\left\vert \varphi (t)\right\vert dt$. We introduce the
weighted space $C_{\mu ,\log }[I,E]$ of continuous functions $\varphi $ on $%
I $ as follows 
\begin{equation*}
C_{1-\gamma ,\log }[I,E]=\{{\varphi :(1,b]\rightarrow E:[\log (}t{%
)]^{1-\gamma }\varphi (t)\in C[I,E]\},}\text{ }0\leq \gamma <1,
\end{equation*}%
with the norm 
\begin{equation*}
\Vert \varphi \Vert _{C_{1-\gamma ,\log }}=\big\|{[\log (}t{)]^{1-\gamma
}\varphi (t)}\big\|_{C}=\sup \{\left\vert {[\log (}t{)]^{1-\gamma }\varphi
(t)}\right\vert ;\text{ }t\in I\}.
\end{equation*}

Obviously, $C_{1-\gamma ,\log }[I,E]$ is Banach space with the norm $\Vert
\cdot \Vert _{C_{1-\gamma ,\log }}.$ Meanwhile, 
\begin{equation*}
C_{1-\gamma ,\log }^{1}[I,E]:=\{\varphi \in C[I,E]:\varphi ^{(1)}\in
C_{1-\gamma ,\log }[I,E]\}
\end{equation*}
is the Banach space with the norm 
\begin{equation*}
\Vert {\varphi }\Vert _{C_{1-\gamma ,\log }^{1}}=\Vert {\varphi }\Vert
_{C}+\Vert {\varphi }^{(1)}\Vert _{C_{1-\gamma ,\log }}.
\end{equation*}%
Moreover, $C_{1-\gamma ,\log }^{0}[I,E]:=C_{1-\gamma ,\log }[I,E].$

\begin{definition}
\cite{AB2} The Hadamard fraction integral of order $\alpha >0$ of a function 
${\varphi \in }L^{1}[I,E]$ is defined by 
\begin{equation*}
\big(_{H}I_{1^{+}}^{\alpha }{\varphi }\big)(t)=\frac{1}{\Gamma (\alpha )}%
\int_{1}^{t}\bigg(\log \frac{t}{s}\bigg)^{\alpha -1}{\varphi }(s)\frac{ds}{s}%
,
\end{equation*}%
provided the integral exits.
\end{definition}

\begin{definition}
\cite{JYM26} Let $n-1<\alpha <n$ and ${\varphi }:[1,\infty ]\rightarrow E$
is a continuous function. The Hadamard fractional derivative of order $%
\alpha $ of ${\varphi }$ is defined by 
\begin{equation*}
\big(_{H}D_{1^{+}}^{\alpha }{\varphi }\big)(t)=\frac{1}{\Gamma (n-\alpha )}%
\bigg(t\frac{d}{dt}\bigg)^{n}\int_{1}^{t}\bigg(\log \frac{t}{s}\bigg)%
^{n-\alpha -1}{\varphi }(s)\frac{ds}{s},
\end{equation*}%
where $n=\lceil \alpha \rceil +1,$ $\lceil \alpha \rceil $ denotes the
integer part of real number $\alpha $, and $\log (.)=\log _{e}(.).$
\end{definition}

\begin{definition}
\label{de2} \cite{MN15} Let $n-1<\alpha <n$ and $0\leq \beta \leq 1.$ The
Hilfer-Hadamard fractional derivative of order $\alpha $ of a continuous
function ${\varphi }:[1,\infty ]\rightarrow E$ is defined by 
\begin{equation}
\big(_{H}D_{1^{+}}^{\alpha ,\beta }{\varphi }\big)(t)=\big(%
_{H}I_{1^{+}}^{\beta (n-\alpha )}D^{n}\big(_{H}I_{1^{+}}^{(1-\beta
)(n-\alpha )}{\varphi }\big)\big)(t),  \label{u}
\end{equation}%
where $D^{n}:=\left( \frac{d}{dt}\right) ^{n}.$ One has%
\begin{equation*}
\big(_{H}D_{1^{+}}^{\alpha ,\beta }{\varphi }\big)(t)=\big(%
_{H}I_{1^{+}}^{\beta (n-\alpha )}\big(\ _{H}D_{1^{+}}^{\gamma }{\varphi }%
\big)\big)(t),\ \ \gamma =\alpha +n\beta -\alpha \beta ,
\end{equation*}%
here $_{H}D_{1^{+}}^{\gamma }=\left( \frac{d}{dt}\right) ^{n}$ $%
_{H}I_{1^{+}}^{(1-\beta )(n-\alpha )}=D^{n}$ $_{H}I_{1^{+}}^{n-\gamma }.$
\end{definition}

\begin{remark}
Let $0<\alpha <1,$ $0\leq \beta \leq 1,$ and $\gamma =\alpha +\beta -\alpha
\beta .$

\begin{enumerate}
\item The operator $_{H}D_{1^{+}}^{\alpha ,\beta }$ given by (\ref{u}) can
be rewritten as 
\begin{equation*}
_{H}D_{1^{+}}^{\alpha ,\beta }=\text{ }_{H}I_{1^{+}}^{\beta (1-\alpha )}D%
\text{ }_{H}I_{1^{+}}^{(1-\beta )(1-\alpha )}=\text{ }_{H}I_{1^{+}}^{\beta
(1-\alpha )}\ _{H}D_{1^{+}}^{\gamma },
\end{equation*}%
where $_{H}D_{1^{+}}^{\gamma }=\frac{d}{dt}$ $_{H}I_{1^{+}}^{(1-\beta
)(1-\alpha )}=D$ $_{H}I_{1^{+}}^{1-\gamma }.$

\item The generalization (\ref{u}) for $\beta =0$ coincides with the
Hadamard\ Riemann-Liouville derivative and for $\beta =1$ with the
Hadamard-Caputo derivative, i.e. $_{H}D_{1^{+}}^{\alpha ,0}=$ $D$ $%
_{H}I_{1^{+}}^{1-\alpha }=$ $_{H}D_{1^{+}}^{\alpha }\ $and $%
_{H}D_{1^{+}}^{\alpha ,1}=$ $_{H}I_{1^{+}}^{1-\alpha }D=$ $%
_{_{H}}^{^{C}}D_{1^{+}}^{\alpha }.$
\end{enumerate}
\end{remark}

\begin{lemma}
\label{lem1} \cite{MN15} If $\alpha ,\beta >0$ and $1<t<\infty $, then 
\begin{equation*}
\big(_{H}I_{1^{+}}^{\alpha }(\log t)^{\beta -1}\big)(t)=\frac{\Gamma (\beta )%
}{\Gamma (\beta +\alpha )}\big(\log (t)\big)^{\beta +\alpha -1},
\end{equation*}%
and 
\begin{equation*}
\big(_{H}D_{1^{+}}^{\alpha }(\log t)^{\beta -1}\big)(t)=\frac{\Gamma (\beta )%
}{\Gamma (\beta -\alpha )}\big(\log (t)\big)^{\beta -\alpha -1}.
\end{equation*}%
In particular, if $\beta =1$ and $0<\alpha <1$, then the Hadamard fractional
derivative of a constant is not equal to zero: 
\begin{equation*}
\big(_{H}D_{1^{+}}^{\alpha }1\big)(t)=\frac{1}{\Gamma (1-\alpha )}(\log
t)^{-\alpha },\quad 0<\alpha <1.
\end{equation*}
\end{lemma}

\begin{lemma}
\cite{AHJ16} If $\alpha ,\beta >0$ and ${\varphi }\in L^{1}[I,E]$ for $t\in
\lbrack 1,b]$ there exist the following properties: 
\begin{equation*}
\big(_{H}I_{1^{+}}^{\alpha }\,_{H}I_{1^{+}}^{\beta }{\varphi }\big)(t)=\big(%
_{H}I_{1^{+}}^{\alpha +\beta }{\varphi }\big)(t),
\end{equation*}%
and 
\begin{equation*}
\big(_{H}D_{1^{+}}^{\alpha }\,_{H}I_{1^{+}}^{\alpha }{\varphi }\big)(t)={%
\varphi }(t).
\end{equation*}%
In particular, if ${\varphi }\in C_{\gamma ,\log }[I,E]$ or ${\varphi }\in
C[I,E]$, then these equalities hold at $t\in (1,b]$ or $t\in \lbrack 1,b]$,
respectively.
\end{lemma}

Now, we introduce spaces that helps us to solve and reduce problem FIDE (\ref%
{eq1})-(\ref{eq2}) to an equivalent the mixed type integral equation: 
\begin{equation*}
C_{1-\gamma ,\log }^{\alpha ,\beta }[I,E]=\{f\in C_{1-\gamma ,\log }[I,E],%
\text{ }_{H}D_{1^{+}}^{\alpha ,\beta }f\in C_{1-\gamma ,\log }[I,E]\},\quad
0\leq \gamma <1,
\end{equation*}%
and 
\begin{equation*}
C_{1-\gamma ,\log }^{\gamma }[I,E]=\{f\in C_{1-\gamma ,\log }[I,E],\text{ }%
_{H}D_{1^{+}}^{\gamma }f\in C_{1-\gamma ,\log }[I,E]\},\quad 0\leq \gamma <1.
\end{equation*}%
It is obvious that 
\begin{equation*}
C_{1-\gamma ,\log }^{\gamma }[I,E]\subset C_{1-\gamma ,\log }^{\alpha ,\beta
}[I,E].
\end{equation*}

\begin{lemma}
\label{lem4} \cite{AB2} Let $0<\alpha <1,$ $0\leq \gamma <1$. If $f\in
C_{\gamma ,\log }[I,E]$ and $_{H}I_{1^{+}}^{1-\alpha }{\varphi }\in
C_{\gamma ,\log }^{1}[I,E]$, then 
\begin{equation*}
\big(_{H}I_{1^{+}}^{\alpha }\,_{H}D_{1^{+}}^{\alpha }{\varphi }\big)(t)={%
\varphi }(t)-\frac{(_{H}I_{1^{+}}^{1-\alpha }{\varphi })(1)}{\Gamma (\alpha )%
}(\log t)^{\alpha -1},\text{ \ }\forall t\in (1,b].
\end{equation*}
\end{lemma}

\begin{lemma}
\label{lem2} \cite{VKE} If $0\leq \gamma <1$ and ${\varphi }\in C_{\gamma
,\log }[I,E]$, then 
\begin{equation*}
\big(_{H}I_{1^{+}}^{\alpha }{\varphi }\big)(1):=\lim_{t\rightarrow
1^{+}}\,_{H}I_{1^{+}}^{\alpha }{\varphi }(t)=0,\ \ 0\leq \gamma <\alpha .
\end{equation*}
\end{lemma}

\begin{lemma}
\label{lem3} \cite{VKE} Let $\alpha ,\beta >0$ and $\gamma =\alpha +\beta
-\alpha \beta $. If ${\varphi }\in C_{1-\gamma ,\log }^{\gamma }[I,E]$, then 
\begin{equation*}
_{H}I_{1^{+}}^{\gamma }\,_{H}D_{1^{+}}^{\gamma }{\varphi }=\text{ }%
_{H}I_{1^{+}}^{\alpha }\,_{H}D_{1^{+}}^{\alpha ,\beta }{\varphi },\ \ \
_{H}D_{1^{+}}^{\gamma }\,_{H}I_{1^{+}}^{\alpha }{\varphi }=\text{ }%
_{H}D_{1^{+}}^{\beta (1-\alpha )}{\varphi .}
\end{equation*}
\end{lemma}

\begin{lemma}
\cite{VKE} Let ${\varphi }\in L^{1}[I,E]$ and $_{H}D_{1^{+}}^{\beta
(1-\alpha )}{\varphi }\in L^{1}[I,E]$ exists, then 
\begin{equation*}
_{H}D_{1^{+}}^{\alpha ,\beta }\,_{H}I_{1^{+}}^{\alpha }{\varphi }=\text{ }%
_{H}D_{1^{+}}^{\beta (1-\alpha )}\,{\varphi .}
\end{equation*}
\end{lemma}

\begin{lemma}
\label{le3}(\cite{JYM26} generalized Gronwall inequality) Let $%
v,w:I\rightarrow \lbrack 1,+\infty )$ be continuous functions. If $w$ is
nondecreasing and there are constants $k>0$ and $0<\alpha <1$ such that 
\begin{equation*}
v(t)\leq w(t)+k\int_{1}^{t}\bigg(\log \frac{t}{s}\bigg)^{\alpha -1}v(t)\frac{%
ds}{s},\quad t\in I,
\end{equation*}%
then 
\begin{equation*}
v(t)\leq w(t)+\int_{1}^{t}\Bigg(\sum_{n=1}^{\infty }\frac{(k\Gamma (\alpha
))^{n}}{\Gamma (n\alpha )}\bigg(\log \frac{t}{s}\bigg)^{n\alpha -1}w(t)\Bigg)%
\frac{ds}{s},\quad t\in I.
\end{equation*}
\end{lemma}

\begin{remark}
\label{rem3}In particular, if $w(t)$ be a nondecreasing function on $I$.
then we have 
\begin{equation*}
v(t)\leq w(t)E_{\alpha }(k\Gamma (\alpha )(\log t)^{\alpha }),
\end{equation*}%
where $E_{\alpha }$ is the Mittag-Leffler function defined by 
\begin{equation*}
E_{\alpha }(z)=\sum_{k=0}^{\infty }\frac{z^{k}}{\Gamma (k\alpha +1)},\quad
z\in \mathbb{C}\text{.}
\end{equation*}
\end{remark}

\section{Existence and uniqueness results\label{se3}}

In this section, we prove the existence and uniqueness of solution of the
problem FIDE (\ref{eq1})-(\ref{eq2}).

\begin{lemma}
\label{le1}Let $f:I\times E\times E\rightarrow E$ be a function such that $%
f(.,u(.),_{H}D_{1^{+}}^{\alpha ,\beta }u(.))\in C_{1-\gamma ,\log }[I,E]$
for any $u\in C_{1-\gamma ,\log }[I,E].$ A function $u\in C_{1-\gamma ,\log
}^{\gamma }[I,E]$ is a solution of the following Hilfer-Hadmard FIDE 
\begin{align*}
_{H}D_{1^{+}}^{\alpha ,\beta }u(t)& =f(t,u(t),_{H}D_{1^{+}}^{\alpha ,\beta
}u(t))\qquad 0<\alpha <1,0\leq \beta \leq 1,\quad t\in I, \\
_{H}I_{1^{+}}^{1-\gamma }u(1)& =u_{0},\qquad \gamma =\alpha +\beta -\alpha
\beta ,
\end{align*}%
if and only if $u$ satisfies the following integral equation: 
\begin{equation*}
u(t)=\frac{(\log t)^{\gamma -1}}{\Gamma (\gamma )}u_{0}+\frac{1}{\Gamma
(\alpha )}\int_{1}^{t}\bigg(\log \frac{t}{s}\bigg)^{\alpha
-1}f(s,u(s),_{H}D_{1^{+}}^{\alpha ,\beta }u(s))\frac{ds}{s}.
\end{equation*}
\end{lemma}

\begin{lemma}
\label{lem5}Let $f:I\times E\times E\rightarrow E$ be a function such that $%
f(.,u(.),_{H}D_{1^{+}}^{\alpha ,\beta }u(.))\in C_{1-\gamma ,\log }[I,E]$
for any $u\in C_{1-\gamma ,\log }[I,E].$ A function $u\in C_{1-\gamma ,\log
}^{\gamma }[I,E]$ is a solution of the problem FIDE (\ref{eq1})-(\ref{eq2})
if and only if $u$ satisfies the mixed-type integral equation 
\begin{eqnarray}
u(t) &=&\frac{\phi }{c_{1}+c_{2}}\frac{(\log t)^{\gamma -1}}{\Gamma (\gamma )%
}-\frac{c_{2}}{c_{1}+c_{2}}\frac{(\log t)^{\gamma -1}}{\Gamma (\gamma )}%
\frac{1}{\Gamma (1-\gamma +\alpha )}  \notag \\
&&\int_{1}^{b}\bigg(\log \frac{b}{s}\bigg)^{\alpha -\gamma }\mathcal{F}%
_{u}(s)\frac{ds}{s}\bigg)+\frac{1}{\Gamma (\alpha )}\int_{1}^{t}\bigg(\log 
\frac{t}{s}\bigg)^{\alpha -1}\mathcal{F}_{u}(s)\frac{ds}{s}.  \label{eqq2}
\end{eqnarray}%
where $\mathcal{F}_{u}(t):=f(t,u(t),\mathcal{F}_{u}(t))=$ $%
_{H}D_{1^{+}}^{\alpha ,\beta }u(t).$
\end{lemma}

\begin{proof}
According to Lemma \ref{le1}, a solution of Eq.(\ref{eq1}) can be expressed
by 
\begin{equation}
u(t)=\frac{_{H}I_{1^{+}}^{1-\gamma }u(1^{+})}{\Gamma (\gamma )}(\log
t)^{\gamma -1}+\frac{1}{\Gamma (\alpha )}\int_{1}^{t}\bigg(\log \frac{t}{s}%
\bigg)^{\alpha -1}\mathcal{F}_{u}(s)\frac{ds}{s}.  \label{eq3}
\end{equation}%
Applying $_{H}I_{1^{+}}^{1-\gamma }$ on both sides of Eq.(\ref{eq3}) and
using Lemma \ref{lem1}, it follows that%
\begin{equation}
_{H}I_{1^{+}}^{1-\gamma }u({t)}=\text{ }_{H}I_{1^{+}}^{1-\gamma }u(1^{+})+%
\frac{1}{\Gamma (1-\gamma +\alpha )}\int_{1}^{t}\bigg(\log \frac{t}{s}\bigg)%
^{\alpha -\gamma }\mathcal{F}_{u}(s)\frac{ds}{s}.  \label{eq4}
\end{equation}%
Taking the limit as $t\rightarrow \bar{b}$ on both sides of Eq.(\ref{eq4}),
we get 
\begin{equation}
_{H}I_{1^{+}}^{1-\gamma }u(b^{-})=\text{ }_{H}I_{1^{+}}^{1-\gamma }u(1^{+})+%
\frac{1}{\Gamma (1-\gamma +\alpha )}\int_{1}^{b}\bigg(\log \frac{b}{s}\bigg)%
^{\alpha -\gamma }\mathcal{F}_{u}(s)\frac{ds}{s}.  \label{eq5}
\end{equation}%
From the boundary condition and Eq.(\ref{eq5}), we have%
\begin{eqnarray}
_{H}I_{1^{+}}^{1-\gamma }u(1^{+}) &=&\frac{\phi }{c_{1}+c_{2}}-\frac{c_{2}}{%
c_{1}+c_{2}}  \notag \\
&&\frac{1}{\Gamma (1-\gamma +\alpha )}\int_{1}^{b}\bigg(\log \frac{b}{s}%
\bigg)^{\alpha -\gamma }\mathcal{F}_{u}(s)\frac{ds}{s}.  \label{eq6}
\end{eqnarray}%
Substitute Eq.(\ref{eq6}) into Eq.(\ref{eq3}), it follows that 
\begin{eqnarray}
u(t) &=&\frac{\phi }{c_{1}+c_{2}}\frac{(\log t)^{\gamma -1}}{\Gamma (\gamma )%
}-\frac{c_{2}}{c_{1}+c_{2}}\frac{(\log t)^{\gamma -1}}{\Gamma (\gamma )}%
\frac{1}{\Gamma (1-\gamma +\alpha )}  \notag \\
&&\int_{1}^{b}\bigg(\log \frac{b}{s}\bigg)^{\alpha -\gamma }\mathcal{F}%
_{u}(s)\frac{ds}{s}+\frac{1}{\Gamma (\alpha )}\int_{1}^{t}\bigg(\log \frac{t%
}{s}\bigg)^{\alpha -1}\mathcal{F}_{u}(s)\frac{ds}{s}.  \label{eq7}
\end{eqnarray}%
Thus, $u$ is a solution of the problem FIDE (\ref{eq1})-(\ref{eq2}).

Next, we are ready to prove that $u$ is also a solution of the mixed type
integral equation Eq.(\ref{eqq2}). To this end, by applying the Hadamard
fractional integral $_{H}I_{1^{+}}^{1-\gamma }$ to both sides of Eq.(\ref%
{eqq2})$,$ we can obtain%
\begin{eqnarray}
_{H}I_{1^{+}}^{1-\gamma }u(t) &=&\frac{\phi }{c_{1}+c_{2}}\text{ }-\frac{%
c_{2}}{c_{1}+c_{2}}\frac{1}{\Gamma (1-\gamma +\alpha )}  \notag \\
&&\times \int_{1}^{b}\bigg(\log \frac{b}{s}\bigg)^{\alpha -\gamma }\mathcal{F%
}_{u}(s)\frac{ds}{s}  \notag \\
&&+\frac{1}{\Gamma (1-\beta (1-\alpha )}\int_{1}^{t}\bigg(\log \frac{t}{s}%
\bigg)^{\alpha -\gamma }\mathcal{F}_{u}(s)\frac{ds}{s},  \label{eq8}
\end{eqnarray}%
by multiplying $c_{1}$ and taking the limit as $t\longrightarrow 1^{+}$ into
the above equation, it follows from Lemma \ref{lem2} with $1-\gamma \leq
1-\beta (1-\alpha ),$ that%
\begin{eqnarray}
_{H}I_{1^{+}}^{1-\gamma }c_{1}u(1^{+}) &=&\frac{c_{1}\phi }{c_{1}+c_{2}}%
\text{ }-\frac{c_{2}c_{1}}{c_{1}+c_{2}}\frac{1}{\Gamma (1-\gamma +\alpha )} 
\notag \\
&&\times \int_{1}^{b}\bigg(\log \frac{b}{s}\bigg)^{\alpha -\gamma }\mathcal{F%
}_{u}(s)\frac{ds}{s}\bigg).  \label{eq9}
\end{eqnarray}%
Now, by multiplying $c_{2}$ and taking the limit as $t\longrightarrow b^{-}$
to both sides of Eq.(\ref{eq8})$,$ we conclude that 
\begin{eqnarray}
_{H}I_{1^{+}}^{1-\gamma }c_{2}u(b^{-}) &=&\frac{c_{2}\phi }{c_{1}+c_{2}}%
\text{ }-\frac{c_{2}^{2}}{c_{1}+c_{2}}\frac{1}{\Gamma (1-\gamma +\alpha )} 
\notag \\
&&\times \int_{1}^{b}\bigg(\log \frac{b}{s}\bigg)^{\alpha -\gamma }\mathcal{F%
}_{u}(s)\frac{ds}{s}\bigg)  \notag \\
&&+\frac{c_{2}}{\Gamma (1-\gamma +\alpha )}\int_{1}^{b}\bigg(\log \frac{b}{s}%
\bigg)^{\alpha -\gamma }\mathcal{F}_{u}(s)\frac{ds}{s}.  \label{eq10}
\end{eqnarray}%
By compiling Eq.(\ref{eq9}) and Eq.(\ref{eq10}), we find that 
\begin{eqnarray*}
&&_{H}I_{1^{+}}^{1-\gamma }c_{1}u(1^{+})+_{H}I_{1^{+}}^{1-\gamma
}c_{2}u(b^{-}) \\
&=&\frac{c_{1}\phi }{c_{1}+c_{2}}\text{ }-\frac{c_{2}c_{1}}{c_{1}+c_{2}}%
\frac{1}{\Gamma (1-\gamma +\alpha )}\int_{1}^{b}\bigg(\log \frac{b}{s}\bigg)%
^{\alpha -\gamma }\mathcal{F}_{u}(s)\frac{ds}{s} \\
&&+\frac{c_{2}\phi }{c_{1}+c_{2}}\text{ }-\frac{c_{2}^{2}}{c_{1}+c_{2}}\frac{%
1}{\Gamma (1-\gamma +\alpha )}\int_{1}^{b}\bigg(\log \frac{b}{s}\bigg)%
^{\alpha -\gamma }\mathcal{F}_{u}(s)\frac{ds}{s} \\
&&+\frac{c_{2}}{\Gamma (1-\gamma +\alpha )}\int_{1}^{b}\bigg(\log \frac{b}{s}%
\bigg)^{\alpha -\gamma }\mathcal{F}_{u}(s)\frac{ds}{s} \\
&=&\frac{c_{1}\phi }{c_{1}+c_{2}}+\frac{c_{2}\phi }{c_{1}+c_{2}}-\left( 
\frac{c_{2}c_{1}}{c_{1}+c_{2}}+\frac{c_{2}^{2}}{c_{1}+c_{2}}-c_{2}\right) 
\frac{1}{\Gamma (1-\gamma +\alpha )} \\
&&\times \int_{1}^{b}\bigg(\log \frac{b}{s}\bigg)^{\alpha -\gamma }\mathcal{F%
}_{u}(s)\frac{ds}{s} \\
&=&\phi .
\end{eqnarray*}%
which means that the boundary condition Eq.(\ref{eq2}) is satisfied.

Now, we multiply $_{H}D_{1+}^{\gamma }$ on both sides of Eq.(\ref{eqq2})$,$
it follows from Lemmas \ref{lem1} and \ref{lem3} that 
\begin{equation}
_{H}D_{1^{+}}^{\gamma }u(t)=_{H}D_{1^{+}}^{\beta (1-\alpha )}\mathcal{F}%
_{u}(t)=_{H}D_{1^{+}}^{\beta (1-\alpha )}f(t,u(t),_{H}D_{1^{+}}^{\alpha
,\beta }u(t)).  \label{eq11}
\end{equation}%
Since $u\in C_{1-\gamma ,\log }^{\gamma }[I,E]$ and by the definition of $%
C_{1-\gamma ,\log }^{\gamma }[I,E]$, we have $_{H}D_{1^{+}}^{\gamma }u\in
C_{1-\gamma ,\log }[I,E]$, and $_{H}D_{1^{+}}^{\beta (1-\alpha
)}f=D\,_{H}I_{1^{+}}^{1-\beta (1-\alpha )}f\in C_{1-\gamma ,\log }[I,E].$
For $f\in C_{1-\gamma ,\log }[I,E]$ it is obvious that $_{H}I_{1^{+}}^{1-%
\beta (1-\alpha )}f\in C_{1-\gamma ,\log }[I,E]$, consequently, $%
_{H}I_{1^{+}}^{1-\beta (1-\alpha )}f\in C_{1-\gamma ,\log }^{1}[I,E]$. Thus, 
$f$ and $_{H}I_{1^{+}}^{1-\beta (1-\alpha )}f$ satisfy the conditions of
Lemma \ref{lem4}. Next, by applying \ $_{H}I_{1^{+}}^{\beta (1-\alpha )}$ to
both sides of Eq.(\ref{eq11}), we have from Lemma \ref{lem4}, and Definition %
\ref{de2} that 
\begin{eqnarray*}
\text{ }_{H}D_{1^{+}}^{\alpha ,\beta }u(t) &=&\text{ }_{H}I_{1^{+}}^{\beta
(1-\alpha )}\text{ }_{H}D_{1^{+}}^{\beta (1-\alpha )}\mathcal{F}_{u}(t), \\
&=&\mathcal{F}_{u}(t)-\frac{(_{H}I_{1^{+}}^{1-\beta (1-\alpha )}\mathcal{F}%
_{u})(1)}{\Gamma (\beta (1-\alpha )}(\log t)^{\beta (1-\alpha )-1},
\end{eqnarray*}%
where $(_{H}I_{1^{+}}^{1-\beta (1-\alpha )}\mathcal{F}_{u})(1)=0$ is implied
by Lemma \ref{lem2}. Hence, it reduces to $_{H}D_{1^{+}}^{\alpha ,\beta
}u(t)=\mathcal{F}_{u}(t)=f(t,u(t),_{H}D_{1^{+}}^{\alpha ,\beta }u(t)).$ This
completes the proof.
\end{proof}

Now we use Lemma \ref{lem5} to study the existence and uniqueness of
solution to problem FIDE (\ref{eq1})-(\ref{eq2}). First, we list the
following hypotheses:

(H1) The function $f:I\times E\times E\rightarrow E$ is Carath\'{e}odory.

(H2) There exist $\delta ,\sigma ,\rho \in C_{1-\gamma ,\log }[I,E]$ with $%
\delta ^{\ast }=\sup_{t\in I}\delta (t)<1$ such that 
\begin{equation*}
|f(t,u,v)|\leq \delta (t)+\sigma (t)|u|+\rho (t)|v|,\ \ t\in I,\text{ }%
u,v\in E.
\end{equation*}

(H3) There exist two positive constants $K_{f},L_{f}>0$ such that 
\begin{equation*}
|f(t,u,v)-f(t,\bar{u},\bar{v})|\leq K_{f}|u-\bar{u}|+L_{f}|v-\bar{v}|,\ \
u,v,\bar{u},\bar{v}\in E,\text{ }t\in I.
\end{equation*}%
The existence result for the problem FIDE (\ref{eq1})-(\ref{eq2}) will be
proved by using the Schaefer's fixed-point theorem.

\begin{theorem}
\label{TH3.2}Assume that (H1) and (H2) are satisfied. Then the problem FIDE (%
\ref{eq1})-(\ref{eq2}) has at least one solution in $C_{1-\gamma ,\log
}^{\gamma }[I,E]\subset C_{1-\gamma ,\log }^{\alpha ,\beta }[I,E]$ provided
that 
\begin{equation}
\Omega :=\frac{1}{1-\rho ^{\ast }}\left( \left\vert \frac{c_{2}}{c_{1}+c_{2}}%
\right\vert \frac{1}{\Gamma (\gamma )}+1\right) \sigma ^{\ast }\frac{%
\mathcal{B}(\gamma ,\alpha )}{\Gamma (\alpha )}(\log b)^{\alpha }<1,
\label{e1}
\end{equation}%
where $\sigma ^{\ast }=\sup_{t\in I}\sigma (t)$ and $\rho ^{\ast
}=\sup_{t\in I}\rho (t)<1\ $and $\mathcal{B}(\cdot ,\cdot )$ is a beta
function.
\end{theorem}

\begin{proof}
Consider the operator $\mathcal{Q}:C_{1-\gamma ,\log }[I,E]\rightarrow
C_{1-\gamma ,\log }[I,E]$ defined by%
\begin{eqnarray}
\mathcal{Q}u(t) &=&\frac{\phi }{c_{1}+c_{2}}\frac{(\log t)^{\gamma -1}}{%
\Gamma (\gamma )}-\frac{c_{2}}{c_{1}+c_{2}}\frac{(\log t)^{\gamma -1}}{%
\Gamma (\gamma )}\frac{1}{\Gamma (1-\gamma +\alpha )}  \notag \\
&&\times \int_{1}^{b}\bigg(\log \frac{b}{s}\bigg)^{\alpha -\gamma }\mathcal{F%
}_{u}(s)\frac{ds}{s}+\frac{1}{\Gamma (\alpha )}\int_{1}^{t}\bigg(\log \frac{t%
}{s}\bigg)^{\alpha -1}\mathcal{F}_{u}(s)\frac{ds}{s}.  \label{eq12}
\end{eqnarray}%
It is obvious that the operator $\mathcal{Q}$ is well defined. The proof
will be divided into several steps:

\textbf{Step 1}. We show that $\mathcal{Q}$ is continuous.

Let $\{u_{n}\}_{n\in 
\mathbb{N}
}$ be a sequence such that $u_{n}\rightarrow u$ in $C_{1-\gamma ,\log }[I,E]$%
. Then for each $t\in I$, 
\begin{eqnarray*}
&&\left\vert (\log t)^{1-\gamma }\bigg((\mathcal{Q}u_{n})(t)-(\mathcal{Q}%
u)(t)\bigg)\right\vert \\
&\leq &\left\vert \frac{c_{2}}{c_{1}+c_{2}}\right\vert \frac{1}{\Gamma
(\gamma )}\frac{1}{\Gamma (1-\gamma +\alpha )} \\
&&\times \int_{1}^{b}\bigg(\log \frac{b}{s}\bigg)^{\alpha -\gamma }(\log
s)^{\gamma -1}(\log s)^{1-\gamma }\left\vert \mathcal{F}_{u_{n}}(s)-\mathcal{%
F}_{u}(s)\right\vert \frac{ds}{s} \\
&&+\frac{(\log t)^{1-\gamma }}{\Gamma (\alpha )}\int_{1}^{t}\bigg(\log \frac{%
t}{s}\bigg)^{\alpha -1}(\log s)^{\gamma -1}(\log s)^{1-\gamma }\left\vert 
\mathcal{F}_{u_{n}}(s)-\mathcal{F}_{u}(s)\right\vert \frac{ds}{s} \\
&\leq &\left\vert \frac{c_{2}}{c_{1}+c_{2}}\right\vert \frac{(\log
b)^{\alpha }}{\Gamma (\alpha +1)}\Vert \mathcal{F}_{u_{n}}(.)-\mathcal{F}%
_{u}(.)\Vert _{C_{1-\gamma ,\log }}+\frac{\Gamma (\gamma )(\log t)^{\alpha }%
}{\Gamma (\gamma +\alpha )}\Vert \mathcal{F}_{u_{n}}(.)-\mathcal{F}%
_{u}(.)\Vert _{C_{1-\gamma ,\log }} \\
&\leq &\bigg(\left\vert \frac{c_{2}}{c_{1}+c_{2}}\right\vert \frac{1}{\Gamma
(\alpha +1)}+\frac{\mathcal{B}(\gamma ,\alpha )}{\Gamma (\alpha )}\bigg)%
(\log b)^{\alpha }\Vert \mathcal{F}_{u_{n}}(.)-\mathcal{F}_{u}(.)\Vert
_{C_{1-\gamma ,\log }}.
\end{eqnarray*}%
Since $\mathcal{F}_{u}$ is continuous (i.e., $f$ is Carath\'{e}odory), we
conclude that 
\begin{equation*}
\Vert \mathcal{Q}u_{n}-\mathcal{Q}u\Vert |_{C_{1-\gamma ,\log }}\rightarrow
0\quad as\quad n\rightarrow \infty .
\end{equation*}%
\textbf{Step 2.} We show that $\mathcal{Q}$ maps bounded sets into bounded
sets in $C_{1-\gamma ,\log }[I,E]$. Consider the ball $\mathbb{B}_{\lambda
}=\{{u\in C_{1-\gamma ,\log }}${$[I,E]$}${;\Vert u\Vert }_{C_{1-\gamma ,\log
}}{\leq \lambda \}}$, with ${\lambda \geq }\frac{\Lambda }{1-\Omega },$
where $\Omega <1$ and%
\begin{equation*}
\Lambda :=\frac{\phi }{c_{1}+c_{2}}\frac{1}{\Gamma (\gamma )}+\frac{1}{%
1-\rho ^{\ast }}\left( \left\vert \frac{c_{2}}{c_{1}+c_{2}}\right\vert \frac{%
1}{\Gamma (\gamma )}\frac{1}{\Gamma (2-\gamma +\alpha )}+\frac{1}{\Gamma
(\alpha +1)}\right) (\log b)^{1-\gamma +\alpha }.
\end{equation*}

For any $u\in \mathbb{B}_{\lambda }$ and for each $t\in I,$ we have%
\begin{align}
\left\vert (\log t)^{1-\gamma }(\mathcal{Q}u)(t)\right\vert & \leq \frac{%
\phi }{c_{1}+c_{2}}\frac{1}{\Gamma (\gamma )}+\left\vert \frac{c_{2}}{%
c_{1}+c_{2}}\right\vert \frac{1}{\Gamma (\gamma )}  \notag \\
& \times \frac{1}{\Gamma (1-\gamma +\alpha )}\int_{1}^{b}\bigg(\log \frac{b}{%
s}\bigg)^{\alpha -\gamma }|\mathcal{F}_{u}(s)|\frac{ds}{s}  \notag \\
& \quad +\frac{(\log t)^{1-\gamma }}{\Gamma (\alpha )}\int_{1}^{t}\bigg(\log 
\frac{t}{s}\bigg)^{\alpha -1}|\mathcal{F}_{u}(s)|\frac{ds}{s}.  \label{eq13}
\end{align}%
By hypothesis (H2), we can see that 
\begin{align}
|\mathcal{F}_{u}(t)|& =|f(t,u(t),\mathcal{F}_{u}(t)|\leq \delta (t)+\sigma
(t)|u(t)|+\rho (t)|\mathcal{F}_{u}(t)|  \notag \\
& \leq \delta ^{\ast }+\sigma ^{\ast }|u(t)|+\rho ^{\ast }|\mathcal{F}%
_{u}(t)|,  \notag
\end{align}%
which implies%
\begin{equation}
|\mathcal{F}_{u}(t)|\leq \frac{\delta ^{\ast }+\sigma ^{\ast }|u(t)|}{1-\rho
^{\ast }}.  \label{eq14}
\end{equation}%
By submitting Eq.(\ref{eq14}) into Eq.(\ref{eq13}), we deduce that%
\begin{align}
& \left\vert (\log t)^{1-\gamma }(\mathcal{Q}u)(t)\right\vert  \notag \\
& \leq \frac{\phi }{c_{1}+c_{2}}\frac{1}{\Gamma (\gamma )}+\left\vert \frac{%
c_{2}}{c_{1}+c_{2}}\right\vert \frac{1}{\Gamma (\gamma )\left( 1-\rho ^{\ast
}\right) }  \notag \\
& \quad \times \frac{1}{\Gamma (1-\gamma +\alpha )}\int_{1}^{b}\bigg(\log 
\frac{b}{s}\bigg)^{\alpha -\gamma }\bigg(\delta ^{\ast }+\sigma ^{\ast
}|u(s)|\bigg)\frac{ds}{s}  \notag \\
& \quad +\frac{(\log t)^{1-\gamma }}{\left( 1-\rho ^{\ast }\right) \Gamma
(\alpha )}\int_{1}^{t}\bigg(\log \frac{t}{s}\bigg)^{\alpha -1}\bigg(\delta
^{\ast }+\sigma ^{\ast }|u(s)|\bigg)\frac{ds}{s}  \notag \\
& =\frac{\phi }{c_{1}+c_{2}}\frac{1}{\Gamma (\gamma )}+I_{1}+I_{2},
\label{eq15}
\end{align}%
where 
\begin{eqnarray*}
I_{1} &:&=\left\vert \frac{c_{2}}{c_{1}+c_{2}}\right\vert \frac{1}{\Gamma
(\gamma )\left( 1-\rho ^{\ast }\right) }\frac{1}{\Gamma (1-\gamma +\alpha )}
\\
&&\quad \times \int_{1}^{b}\bigg(\log \frac{b}{s}\bigg)^{\alpha -\gamma }%
\bigg(\delta ^{\ast }+\sigma ^{\ast }|u(s)|\bigg)\frac{ds}{s},
\end{eqnarray*}%
and%
\begin{equation*}
I_{2}:=\frac{(\log t)^{1-\gamma }}{\left( 1-\rho ^{\ast }\right) \Gamma
(\alpha )}\int_{1}^{t}\bigg(\log \frac{t}{s}\bigg)^{\alpha -1}\bigg(\delta
^{\ast }+\sigma ^{\ast }|u(s)|\bigg)\frac{ds}{s}.
\end{equation*}%
Now, we estimate $I_{1}$ and $I_{2}$ each item is separate as follows:%
\begin{eqnarray}
I_{1} &\leq &\left\vert \frac{c_{2}}{c_{1}+c_{2}}\right\vert \frac{1}{\Gamma
(\gamma )}\frac{1}{\left( 1-\rho ^{\ast }\right) \Gamma (1-\gamma +\alpha )}
\notag \\
&&\quad \times \int_{1}^{b}\bigg(\log \frac{b}{s}\bigg)^{\alpha -\gamma }%
\bigg(\delta ^{\ast }+\sigma ^{\ast }(\log s)^{\gamma -1}\Vert u\Vert
_{C_{1-\gamma ,\log }}\bigg)\frac{ds}{s}  \notag \\
&\leq &\left\vert \frac{c_{2}}{c_{1}+c_{2}}\right\vert \frac{1}{\Gamma
(\gamma )\left( 1-\rho ^{\ast }\right) }\left( \frac{(\log b)^{1-\gamma
+\alpha }}{\Gamma (2-\gamma +\alpha )}+\sigma ^{\ast }\frac{\Gamma (\gamma )%
}{\Gamma (\alpha +\gamma )}(\log b)^{\alpha }\Vert u\Vert _{C_{1-\gamma
,\log }}\right)  \notag \\
&\leq &\left\vert \frac{c_{2}}{c_{1}+c_{2}}\right\vert \frac{1}{\Gamma
(\gamma )\left( 1-\rho ^{\ast }\right) }\left( \frac{(\log b)^{1-\gamma
+\alpha }}{\Gamma (2-\gamma +\alpha )}+\sigma ^{\ast }\frac{\mathcal{B}%
(\gamma ,\alpha )}{\Gamma (\alpha )}(\log b)^{\alpha }\lambda \right) .
\label{eq16}
\end{eqnarray}%
and%
\begin{eqnarray}
I_{2} &\leq &\frac{(\log t)^{1-\gamma }}{\left( 1-\rho ^{\ast }\right)
\Gamma (\alpha )}\int_{1}^{t}\bigg(\log \frac{t}{s}\bigg)^{\alpha -1}\bigg(%
\delta ^{\ast }+\sigma ^{\ast }(\log s)^{\gamma -1}\Vert u\Vert
_{C_{1-\gamma ,\log }}\bigg)\frac{ds}{s}  \notag \\
&\leq &\frac{1}{1-\rho ^{\ast }}\left( \frac{(\log t)^{1-\gamma +\alpha }}{%
\Gamma (\alpha +1)}+\sigma ^{\ast }\frac{\Gamma (\gamma )}{\Gamma (\alpha
+\gamma )}(\log t)^{\alpha }\Vert u\Vert _{C_{1-\gamma ,\log }}\right) 
\notag \\
&\leq &\frac{1}{1-\rho ^{\ast }}\left( \frac{(\log b)^{1-\gamma +\alpha }}{%
\Gamma (\alpha +1)}+\sigma ^{\ast }\frac{\mathcal{B}(\gamma ,\alpha )}{%
\Gamma (\alpha )}(\log b)^{\alpha }\lambda \right) .  \label{eq17}
\end{eqnarray}%
The Eqs.(\ref{eq16}) and (\ref{eq17}), gives%
\begin{eqnarray*}
I_{1}+I_{2} &\leq &\frac{1}{1-\rho ^{\ast }}\left( \left\vert \frac{c_{2}}{%
c_{1}+c_{2}}\right\vert \frac{1}{\Gamma (\gamma )}\frac{1}{\Gamma (2-\gamma
+\alpha )}+\frac{1}{\Gamma (\alpha +1)}\right) (\log b)^{1-\gamma +\alpha }
\\
&&+\frac{1}{1-\rho ^{\ast }}\left( \left\vert \frac{c_{2}}{c_{1}+c_{2}}%
\right\vert \frac{1}{\Gamma (\gamma )}+1\right) \sigma ^{\ast }\frac{%
\mathcal{B}(\gamma ,\alpha )}{\Gamma (\alpha )}(\log b)^{\alpha }\lambda .
\end{eqnarray*}%
The last equality with Eq.(\ref{eq15}) and definitions of $\Omega ,$ $%
\lambda $ and $\Lambda ,$ implies 
\begin{equation*}
\left\vert (\log t)^{1-\gamma }(\mathcal{Q}u)(t)\right\vert \leq \frac{\phi 
}{c_{1}+c_{2}}\frac{1}{\Gamma (\gamma )}+I_{1}+I_{2}\leq \Lambda +\Omega
\lambda \leq (1-\Omega )\lambda +\Omega \lambda =\lambda .
\end{equation*}

Therefore, $\Vert \mathcal{Q}u\Vert _{C_{1-\gamma ,\log }}\leq \lambda .$
This means that $(\mathcal{Q}\mathbb{B}_{\lambda })$ is uniformly bounded in 
$\mathbb{B}_{\lambda }.$

\textbf{Step 3.} $\mathcal{Q}$ maps bounded sets into equicontinuous set of $%
C_{1-\gamma ,\log }[I,E]$.

Let $u\in \mathbb{B}_{\lambda },$ where $\mathbb{B}_{\lambda }$ is bounded
set defined as in Step 2, and $t_{1},t_{2}\in I,$ with $t_{1}\leq t_{2}$.
Then 
\begin{align*}
& |(\log t_{2})^{1-\gamma }(\mathcal{Q}u)(t_{2})-(\mathcal{Q}u)(t_{1})(\log
t_{1})^{1-\gamma }| \\
& =\left\vert \frac{(\log t_{2})^{1-\gamma }}{\Gamma (\alpha )}%
\int_{1}^{t_{2}}\bigg(\log \frac{t_{2}}{s}\bigg)^{\alpha -1}\mathcal{F}%
_{u}(s)\frac{ds}{s}\right. \\
& -\left. \frac{(\log t_{1})^{1-\gamma }}{\Gamma (\alpha )}\int_{1}^{t_{1}}%
\bigg(\log \frac{t_{1}}{s}\bigg)^{\alpha -1}\mathcal{F}_{u}(s)\frac{ds}{s}%
\right\vert \\
& \leq \left\vert \frac{(\log t_{2})^{1-\gamma }}{\Gamma (\alpha )}%
\int_{1}^{t_{2}}\bigg(\log \frac{t_{2}}{s}\bigg)^{\alpha -1}(\log s)^{\gamma
-1}\max_{s\in I,u\in \mathbb{B}_{\lambda }}\left\vert (\log s)^{1-\gamma }%
\mathcal{F}_{u}(s)\right\vert \frac{ds}{s}\right. \\
& -\left. \frac{(\log t_{1})^{1-\gamma }}{\Gamma (\alpha )}\int_{1}^{t_{1}}%
\bigg(\log \frac{t_{1}}{s}\bigg)^{\alpha -1}(\log s)^{\gamma -1}\max_{s\in
I,u\in \mathbb{B}_{\lambda }}\left\vert (\log s)^{1-\gamma }\mathcal{F}%
_{u}(s)\right\vert \frac{ds}{s}\right\vert \\
& \leq \frac{\mathcal{B}(\gamma ,\alpha )}{\Gamma (\alpha )}\Vert \mathcal{F}%
_{u}\Vert _{C_{1-\gamma ,\log }}\left\vert (\log t_{2})^{\alpha }-(\log
t_{1})^{\alpha }\right\vert .
\end{align*}%
The right-hand side of the above inequality is likely to be zero as $%
t_{1}\rightarrow t_{2}$. As an outcome of Steps 1-3 in concert with
Arzela-Ascoli lemma, it proves that $\mathcal{Q}:C_{1-\gamma ,\log
}[I,E]\rightarrow C_{1-\gamma ,\log }[I,E]$ is continuous and completely
continuous.\newline
\textbf{Step 4. }A priori bounds.

We show that the set $\mathcal{S}$ is bounded, where 
\begin{equation*}
\mathcal{S}=\{u\in C_{1-\gamma ,\log }[I,E]:u=\kappa \mathcal{Q}u,\quad 
\text{for some }0<\kappa <1\}.
\end{equation*}%
Indeed, let $u\in \mathcal{S},$ such that $u=\kappa \mathcal{Q}u$ for some $%
0<\kappa <1$, we have 
\begin{eqnarray*}
\mathcal{Q}u(t) &=&\kappa \mathcal{Q}u(t)=\kappa \bigg(\frac{\phi }{%
c_{1}+c_{2}}\frac{(\log t)^{\gamma -1}}{\Gamma (\gamma )}-\frac{c_{2}}{%
c_{1}+c_{2}}\frac{(\log t)^{\gamma -1}}{\Gamma (\gamma )}\frac{1}{\Gamma
(1-\gamma +\alpha )} \\
&&\int_{1}^{b}\bigg(\log \frac{b}{s}\bigg)^{\alpha -\gamma }\mathcal{F}%
_{u}(s)\frac{ds}{s}\bigg)+\frac{1}{\Gamma (\alpha )}\int_{1}^{t}\bigg(\log 
\frac{t}{s}\bigg)^{\alpha -1}\mathcal{F}_{u}(s)\frac{ds}{s}\bigg).
\end{eqnarray*}%
Then by step 2, for each $t\in I$, 
\begin{eqnarray*}
\Vert \mathcal{Q}u\Vert _{C_{1-\gamma ,\log }} &\leq &\frac{\phi }{%
c_{1}+c_{2}}\frac{1}{\Gamma (\gamma )}+\frac{1}{1-\rho ^{\ast }} \\
&&\times \left( \left\vert \frac{c_{2}}{c_{1}+c_{2}}\right\vert \frac{1}{%
\Gamma (\gamma )}\frac{1}{\Gamma (2-\gamma +\alpha )}+\frac{1}{\Gamma
(\alpha +1)}\right) (\log b)^{1-\gamma +\alpha } \\
&&+\frac{1}{1-\rho ^{\ast }}\left( \left\vert \frac{c_{2}}{c_{1}+c_{2}}%
\right\vert \frac{1}{\Gamma (\gamma )}+1\right) \sigma ^{\ast }\frac{%
\mathcal{B}(\gamma ,\alpha )}{\Gamma (\alpha )}(\log b)^{\alpha }\Vert 
\mathcal{Q}u\Vert _{C_{1-\gamma ,\log }} \\
&\leq &\Lambda +\Omega \Vert \mathcal{Q}u\Vert _{C_{1-\gamma ,\log }}.
\end{eqnarray*}

Since $0<\kappa <1,$ then $u<\mathcal{Q}u,$ it follows by (H2) that 
\begin{equation*}
\Vert u\Vert _{C_{1-\gamma ,\log }}<\Vert \mathcal{Q}u\Vert _{C_{1-\gamma
,\log }}\leq \frac{\Lambda }{1-\Omega }\leq {\lambda .}
\end{equation*}%
This shows that the set $\mathcal{S}$ is bounded. As a consequence of
Schaefer's fixed point theorem, we conclude that $\mathcal{Q}$ has a fixed
point, which is a solution of the problem FIDE (\ref{eq1})-(\ref{eq2}). The
proof is completed.
\end{proof}

The uniqueness result for the problem FIDE (\ref{eq1})-(\ref{eq2}) will be
proved by using the Banach contraction principle.

\begin{theorem}
\label{t3}Assume that (H1) and (H3) are satisfied. If 
\begin{equation}
\mathcal{A}:=\Bigg[\left\vert \frac{c_{2}}{c_{1}+c_{2}}\right\vert \frac{1}{%
\Gamma (\alpha +1)}+\frac{\mathcal{B}(\gamma ,\alpha )}{\Gamma (\alpha )}%
\Bigg](\log b)^{\alpha }\frac{K_{f}}{1-L_{f}}<1,  \label{eq21}
\end{equation}%
then problem FIDE (\ref{eq1})-(\ref{eq2}) has a unique solution.
\end{theorem}

\begin{proof}
Let the operator $\mathcal{Q}:C_{1-\gamma ,\log }[I,E]\rightarrow
C_{1-\gamma ,\log }[I,E]$ is defined by Eq.(\ref{eq12}). It is clear that
the fixed points of $\mathcal{Q}$ are solutions of the problem FIDE (\ref%
{eq1})-(\ref{eq2}). By using the Banach contraction principle, we shall show
that $\mathcal{Q}$ has a fixed point which is the unique solution of the
problem FIDE (\ref{eq1})-(\ref{eq2}). Let $u,v\in C_{1-\gamma ,\log }[I,E]$
and $t\in I$. Then we have 
\begin{eqnarray}
&&\left\vert (\log t)^{1-\gamma }\bigg((\mathcal{Q}u)(t)-(\mathcal{Q}v)(t)%
\bigg)\right\vert  \notag \\
&\leq &\left\vert \frac{c_{2}}{c_{1}+c_{2}}\right\vert \frac{1}{\Gamma
(\gamma )}\frac{1}{\Gamma (1-\gamma +\alpha )}  \notag \\
&&\times \int_{1}^{b}\bigg(\log \frac{b}{s}\bigg)^{\alpha -\gamma
}\left\vert \mathcal{F}_{u}(s)-\mathcal{F}_{v}(s)\right\vert \frac{ds}{s} 
\notag \\
&&+\frac{(\log t)^{1-\gamma }}{\Gamma (\alpha )}\int_{1}^{t}\bigg(\log \frac{%
t}{s}\bigg)^{\alpha -1}\left\vert \mathcal{F}_{u}(s)-\mathcal{F}%
_{v}(s)\right\vert \frac{ds}{s}.  \label{eq18}
\end{eqnarray}%
Since 
\begin{eqnarray*}
\left\vert \mathcal{F}_{u}(s)-\mathcal{F}_{v}(s)\right\vert &=&\left\vert
f(s,u(s),\mathcal{F}_{u}(s))-f(s,v(s),\mathcal{F}_{v}(s))\right\vert \\
&\leq &K_{f}\left\vert u(s)-v(s)\right\vert +L_{f}\left\vert \mathcal{F}%
_{u}(s))-\mathcal{F}_{v}(s))\right\vert ,
\end{eqnarray*}%
which gives%
\begin{equation}
\left\vert \mathcal{F}_{u}(s)-\mathcal{F}_{v}(s)\right\vert \leq \frac{K_{f}%
}{1-L_{f}}\left\vert u(s)-v(s)\right\vert .  \label{U4}
\end{equation}%
It follows, 
\begin{eqnarray}
&&\left\vert (\log t)^{1-\gamma }\bigg((\mathcal{Q}u)(t)-(\mathcal{Q}v)(t)%
\bigg)\right\vert  \notag \\
&\leq &\left\vert \frac{c_{2}}{c_{1}+c_{2}}\right\vert \frac{1}{\Gamma
(\gamma )}\frac{1}{\Gamma (1-\gamma +\alpha )}\int_{1}^{b}\bigg(\log \frac{b%
}{s}\bigg)^{\alpha -\gamma }\frac{K_{f}}{1-L_{f}}\left\vert
u(s)-v(s)\right\vert \frac{ds}{s}\bigg)  \notag \\
&&+\frac{(\log t)^{1-\gamma }}{\Gamma (\alpha )}\int_{1}^{t}\bigg(\log \frac{%
t}{s}\bigg)^{\alpha -1}\frac{K_{f}}{1-L_{f}}\left\vert u(s)-v(s)\right\vert 
\frac{ds}{s}\bigg)  \notag \\
&\leq &\left\vert \frac{c_{2}}{c_{1}+c_{2}}\right\vert \frac{1}{\Gamma
(\gamma )}\frac{K_{f}}{1-L_{f}}\frac{1}{\Gamma (1-\gamma +\alpha )}%
\int_{1}^{b}\bigg(\log \frac{b}{s}\bigg)^{\alpha -\gamma }(\log s)^{\gamma
-1}\Vert u-v\Vert _{C_{1-\gamma ,\log }}  \notag \\
&&+\frac{K_{f}}{1-L_{f}}\frac{(\log t)^{1-\gamma }}{\Gamma (\alpha )}%
\int_{1}^{t}\bigg(\log \frac{t}{s}\bigg)^{\alpha -1}(\log s)^{\gamma
-1}\Vert u-v\Vert _{C_{1-\gamma ,\log }}  \notag \\
&\leq &\left( \left\vert \frac{c_{2}}{c_{1}+c_{2}}\right\vert \frac{1}{%
\Gamma (\alpha +1)}+\frac{\mathcal{B}(\gamma ,\alpha )}{\Gamma (\alpha )}%
\right) \frac{K_{f}}{1-L_{f}}(\log b)^{\alpha }\Vert u-v\Vert _{C_{1-\gamma
,\log }}.  \label{eq19}
\end{eqnarray}%
This implies,%
\begin{equation*}
\Vert \mathcal{Q}u-\mathcal{Q}v\Vert _{C_{1-\gamma ,\log }}\leq \mathcal{A}%
\Vert u-v\Vert _{C_{1-\gamma ,\log }}.
\end{equation*}
\end{proof}

Since $\mathcal{A}<1,$ the operator $\mathcal{Q}$ is a contraction mapping.$%
\ $Hence, by the conclusion of Banach contraction principle the operator $%
\mathcal{Q}$ has a unique fixed point, which is solution of the problem FIDE
(\ref{eq1})-(\ref{eq2}). The proof is completed.

\section{Ulam-Hyers-Rassias stability\label{se4}}

In this section, we discuss our results concerning Hyers--Ulam stablity and
Hyers--Ulam--Rassias stablity of Hilfer-Hadamard-type problem FIDE (\ref{eq1}%
)-(\ref{eq2}). The following observations are taken from \cite{BB4,Ru35}.

\begin{definition}
\label{De1} The problem (\ref{eq1})-(\ref{eq2}) is Ulam-Hyers stable if
there exists a real number $C_{f}>0$ such that for each $\epsilon >0$ and
for each solution $\widetilde{u}\in C_{1-\gamma ,\log }^{\gamma }[I,E]$ of
the inequality 
\begin{equation}
|_{H}D_{1^{+}}^{\alpha ,\beta }\widetilde{u}(t)-\mathcal{F}_{\widetilde{u}%
}(t)|\leq \epsilon ,\qquad t\in I,  \label{16a}
\end{equation}%
there exists a solution $u\in C_{1-\gamma ,\log }^{\gamma }[I,E]$ of the
problem (\ref{eq1})-(\ref{eq2}) such that 
\begin{equation*}
|\widetilde{u}(t)-u(t)|\leq C_{f^{\epsilon }}\epsilon ,\qquad t\in I,
\end{equation*}%
where $\mathcal{F}_{\widetilde{u}}(t)=f(t,\widetilde{u}(t),_{H}D_{1^{+}}^{%
\alpha ,\beta }\widetilde{u}(t)).$
\end{definition}

\begin{definition}
\label{De2}The problem (\ref{eq1})-(\ref{eq2}) is generalized Ulam-Hyers
stable if there exists $\phi _{f}\in C([1,\infty ),[1;\infty ))$ with $\phi
_{f}(1)=0$ such that for each solution $\widetilde{u}\in C_{1-\gamma ,\log
}^{\gamma }[I,E]$ of the inequality 
\begin{equation}
|_{H}D_{1^{+}}^{\alpha ,\beta }\widetilde{u}(t)-\mathcal{F}_{\widetilde{u}%
}(t)|\leq \epsilon ,\qquad t\in I,  \label{16b}
\end{equation}%
there exists a solution $u\in C_{1-\gamma ,\log }^{\gamma }[I,E]$ of the
problem (\ref{eq1})-(\ref{eq2}) such that 
\begin{equation*}
|\widetilde{u}(t)-u(t)|\leq \phi _{f}(\epsilon ),\qquad t\in I.
\end{equation*}
\end{definition}

\begin{definition}
\label{De3}The problem (\ref{eq1})-(\ref{eq2}) is Ulam-Hyers-Rassias stable
if there exists $\varphi \in C_{1-\gamma ,\log }[I,E]$, and there exists a
real number $C_{f,\varphi }>0$ such that for each $\epsilon >0$ and for each
solution $\widetilde{u}\in C_{1-\gamma ,\log }^{\gamma }[I,E]$ of the
inequality 
\begin{equation}
|_{H}D_{1^{+}}^{\alpha ,\beta }\widetilde{u}(t)-\mathcal{F}_{\widetilde{u}%
}(t)|\leq \epsilon \varphi (t),\qquad t\in I,  \label{16c}
\end{equation}%
there exists a solution $u\in C_{1-\gamma ,\log }^{\gamma }[I,E]$ of the
problem (\ref{eq1})-(\ref{eq2}) such that 
\begin{equation*}
|\widetilde{u}(t)-u(t)|\leq C_{f,\varphi }\epsilon \varphi (t),\qquad t\in I.
\end{equation*}
\end{definition}

\begin{definition}
\label{De4}The problem (\ref{eq1})-(\ref{eq2}) is generalized
Ulam-Hyers-Rassias stable if there exists $\varphi \in C_{1-\gamma ,\log
}[I,E]$, and there exists a real number $C_{f,\varphi }>0$ such that for
each solution $\widetilde{u}\in C_{1-\gamma ,\log }^{\gamma }[I,E]$ of the
inequality 
\begin{equation}
|_{H}D_{1^{+}}^{\alpha ,\beta }\widetilde{u}(t)-\mathcal{F}_{\widetilde{u}%
}(t)|\leq \varphi (t),\qquad t\in I,  \label{16d}
\end{equation}%
there exists a solution $u\in C_{1-\gamma ,\log }^{\gamma }[I,E]$ of the
problem (\ref{eq1})-(\ref{eq2}) such that 
\begin{equation*}
|\widetilde{u}(t)-u(t)|\leq C_{f,\varphi }\varphi (t),\qquad t\in I.
\end{equation*}
\end{definition}

\begin{remark}
\label{re3}A function $\widetilde{u}\in C_{1-\gamma ,\log }^{\gamma }[I,E]$
is a solution of the inequality (\ref{16a}) if and only if there exist a
function $h\in C_{1-\gamma ,\log }^{\gamma }[I,E]$ (where $h$ depends on
solution $\widetilde{u}$) such that

\begin{description}
\item[(i)] $|h(t)|\leq \epsilon $ for all $t\in I$,

\item[(ii)] $_{H}D_{1^{+}}^{\alpha ,\beta }\widetilde{u}(t)=\mathcal{F}_{%
\widetilde{u}}(t)+h(t),\quad t\in I$.
\end{description}
\end{remark}

\begin{remark}
\label{re2}It is clear that

\begin{enumerate}
\item Definition \ref{De1} $\Longrightarrow $ Definition \ref{De2}.

\item Definition \ref{De3} $\Longrightarrow $\ Definition \ref{De4}.

\item Definition \ref{De3} with $\varphi (t)$= $1$ $\Longrightarrow $\
Definition \ref{De1}.
\end{enumerate}
\end{remark}

\begin{lemma}
\label{lem6}Let $\widetilde{u}\in C_{1-\gamma ,\log }^{\gamma }[I,E]$ is a
solution of the inequality (\ref{16a}). Then $u$ is a solution of the
following integral inequality:%
\begin{eqnarray*}
&&\left\vert \widetilde{u}(t)-Z_{\widetilde{u}}-\frac{1}{\Gamma (\alpha )}%
\int_{1}^{t}\bigg(\log \frac{t}{s}\bigg)^{\alpha -1}\mathcal{F}_{\widetilde{u%
}}(s)\frac{ds}{s}\right\vert \\
&\leq &\left[ \left\vert \frac{c_{2}}{c_{1}+c_{2}}\right\vert \frac{1}{%
\Gamma (\gamma )}\frac{(\log b)^{\alpha }}{\Gamma (2-\gamma +\alpha )}+\frac{%
(\log b)^{\alpha }}{\Gamma (\alpha +1)}\right] \epsilon .
\end{eqnarray*}%
where%
\begin{eqnarray}
Z_{\widetilde{u}} &=&\frac{\phi }{c_{1}+c_{2}}\frac{(\log t)^{\gamma -1}}{%
\Gamma (\gamma )}-\frac{c_{2}}{c_{1}+c_{2}}\frac{(\log t)^{\gamma -1}}{%
\Gamma (\gamma )}  \notag \\
&&\times \frac{1}{\Gamma (1-\gamma +\alpha )}\int_{1}^{b}\bigg(\log \frac{b}{%
s}\bigg)^{\alpha -\gamma }\mathcal{F}_{\widetilde{u}}(s)\frac{ds}{s}.
\label{w2}
\end{eqnarray}
\end{lemma}

\begin{proof}
In view of Remark \ref{re2}, and Lemma, we have 
\begin{equation*}
_{H}D_{1^{+}}^{\alpha ,\beta }\widetilde{u}(t)=\mathcal{F}_{\widetilde{u}%
}(t)+h(t),
\end{equation*}%
then by utilize Lemma \ref{lem5}, we get%
\begin{eqnarray}
\widetilde{u}(t) &=&\frac{\phi }{c_{1}+c_{2}}\frac{(\log t)^{\gamma -1}}{%
\Gamma (\gamma )}-\frac{c_{2}}{c_{1}+c_{2}}\frac{(\log t)^{\gamma -1}}{%
\Gamma (\gamma )}  \notag \\
&&\times \frac{1}{\Gamma (1-\gamma +\alpha )}\int_{1}^{b}\bigg(\log \frac{b}{%
s}\bigg)^{\alpha -\gamma }\mathcal{F}_{\widetilde{u}}(s)\frac{ds}{s}  \notag
\\
&&-\frac{c_{2}}{c_{1}+c_{2}}\frac{(\log t)^{\gamma -1}}{\Gamma (\gamma )}%
\frac{1}{\Gamma (1-\gamma +\alpha )}\int_{1}^{b}\bigg(\log \frac{b}{s}\bigg)%
^{\alpha -\gamma }h(s)\frac{ds}{s}  \notag \\
&&+\frac{1}{\Gamma (\alpha )}\int_{1}^{t}\bigg(\log \frac{t}{s}\bigg)%
^{\alpha -1}\mathcal{F}_{\widetilde{u}}(s)\frac{ds}{s}  \notag \\
&&+\frac{1}{\Gamma (\alpha )}\int_{1}^{t}\bigg(\log \frac{t}{s}\bigg)%
^{\alpha -1}h(s)\frac{ds}{s}.  \label{w1}
\end{eqnarray}%
From this it follows that%
\begin{eqnarray*}
&&\Bigg |\widetilde{u}(t)-Z_{\widetilde{u}}-\frac{1}{\Gamma (\alpha )}%
\int_{1}^{t}\bigg(\log \frac{t}{s}\bigg)^{\alpha -1}\mathcal{F}_{\widetilde{u%
}}(s)\frac{ds}{s}\Bigg | \\
&\leq &\left\vert \frac{c_{2}}{c_{1}+c_{2}}\right\vert \frac{(\log
t)^{\gamma -1}}{\Gamma (\gamma )}\frac{1}{\Gamma (1-\gamma +\alpha )}%
\int_{1}^{b}\bigg(\log \frac{b}{s}\bigg)^{\alpha -\gamma }\left\vert
h(s)\right\vert \frac{ds}{s} \\
&&+\frac{1}{\Gamma (\alpha )}\int_{1}^{t}\bigg(\log \frac{t}{s}\bigg)%
^{\alpha -1}\left\vert h(s)\right\vert \frac{ds}{s}\Bigg | \\
&\leq &\left( \left\vert \frac{c_{2}}{c_{1}+c_{2}}\right\vert \frac{1}{%
\Gamma (\gamma )}\frac{(\log b)^{\alpha }}{\Gamma (2-\gamma +\alpha )}+\frac{%
(\log b)^{\alpha }}{\Gamma (\alpha +1)}\right) \epsilon .
\end{eqnarray*}%
The same arguments one can deal with the solutions of inequalities (\ref{16c}%
) and (\ref{16d}).
\end{proof}

Now we give the Ulam-Hyers and Ulam-Hyers-Rassias results in this sequel.

\begin{theorem}
\label{t1}Assume that the assumptions of Theorem \ref{t3} are satisfied.
Then the problem FIDE (\ref{eq1})-(\ref{eq2}) is Ulam-Hyers stable.
\end{theorem}

\begin{proof}
Let $\epsilon >0$, and let $\widetilde{u}\in C_{1-\gamma ,\log }^{\gamma
}[I,E]$ be a function, which satisfies the inequality (\ref{16a}), and let $%
u\in C_{1-\gamma ,\log }^{\gamma }[I,E]$ is a unique solution of the
following Hilfer-Hadamard-type 
\begin{equation}
_{H}D_{1^{+}}^{\alpha ,\beta }u(t)=\mathcal{F}_{u}(t),\qquad t\in
J=[1,b],\qquad \qquad \qquad \qquad  \label{n1}
\end{equation}%
\begin{equation}
_{H}I_{1^{+}}^{1-\gamma }\left[ c_{1}u(1^{+})+c_{2}u(b^{-})\right] =\text{ }%
_{H}I_{1^{+}}^{1-\gamma }\left[ c_{1}\widetilde{u}(1^{+})+c_{2}\widetilde{u}%
(b^{-})\right] =\phi ,  \label{n2}
\end{equation}%
where $0<\alpha <1,\quad 0\leq \beta \leq 1$, $\gamma =\alpha +\beta -\alpha
\beta ,$ and $c_{1}+c_{2}\neq 0$ $($ $c_{2}\neq 0).$ Using Lemma \ref{lem5},
we obtain 
\begin{equation*}
u(t)=Z_{u}+\frac{1}{\Gamma (\alpha )}\int_{1}^{t}\bigg(\log \frac{t}{s}\bigg)%
^{\alpha -1}\mathcal{F}_{u}(s)\frac{ds}{s},
\end{equation*}%
where 
\begin{eqnarray}
Z_{u} &=&\frac{\phi }{c_{1}+c_{2}}\frac{(\log t)^{\gamma -1}}{\Gamma (\gamma
)}-\frac{c_{2}}{c_{1}+c_{2}}\frac{(\log t)^{\gamma -1}}{\Gamma (\gamma )} 
\notag \\
&&\times \frac{1}{\Gamma (1-\gamma +\alpha )}\int_{1}^{b}\bigg(\log \frac{b}{%
s}\bigg)^{\alpha -\gamma }\mathcal{F}_{u}(s)\frac{ds}{s}.  \label{w3}
\end{eqnarray}%
By integration of the inequality (\ref{16a}) and applying Lemma \ref{lem6},
we obtain 
\begin{equation}
\Bigg|\widetilde{u}(t)-Z_{\widetilde{u}}-\frac{1}{\Gamma (\alpha )}%
\int_{1}^{t}\big(\log \frac{t}{s}\big)^{\alpha -1}\mathcal{F}_{\widetilde{u}%
}(s)\frac{ds}{s}\Bigg|\leq B\epsilon ,  \label{U3}
\end{equation}%
where $B=\left\vert \frac{c_{2}}{c_{1}+c_{2}}\right\vert \frac{1}{\Gamma
(\gamma )}\frac{(\log b)^{\alpha }}{\Gamma (2-\gamma +\alpha )}+\frac{(\log
b)^{\alpha }}{\Gamma (\alpha +1)}$. If the condition (\ref{n2})$\ $holds, it
follows that%
\begin{equation*}
c_{1}\text{ }_{H}I_{1^{+}}^{1-\gamma }\left[ u(1^{+})-\widetilde{u}(1^{+})%
\right] =c_{2}\text{ }_{H}I_{1^{+}}^{1-\gamma }\left[ \widetilde{u}%
(b^{-})-u(b^{-})\right] ,
\end{equation*}%
and $Z_{\widetilde{u}}=Z_{u}.$ Indeed,%
\begin{eqnarray*}
&&\left\vert Z_{\widetilde{u}}-Z_{u}\right\vert \\
&\leq &\left\vert \frac{c_{2}}{c_{1}+c_{2}}\right\vert \frac{(\log
t)^{\gamma -1}}{\Gamma (\gamma )}\frac{1}{\Gamma (1-\gamma +\alpha )}%
\int_{1}^{b}\bigg(\log \frac{b}{s}\bigg)^{\alpha -\gamma }\left\vert 
\mathcal{F}_{\widetilde{u}}(s)-\mathcal{F}_{u}(s)\right\vert \frac{ds}{s} \\
&\leq &\left\vert \frac{c_{2}}{c_{1}+c_{2}}\right\vert \frac{(\log
t)^{\gamma -1}}{\Gamma (\gamma )}\text{ }_{H}I_{1^{+}}^{1-\gamma +\alpha
}\left\vert \mathcal{F}_{\widetilde{u}}(b)-\mathcal{F}_{u}(b)\right\vert \\
&\leq &\left\vert \frac{c_{2}}{c_{1}+c_{2}}\right\vert \frac{(\log
t)^{\gamma -1}}{\Gamma (\gamma )}\frac{K_{f}}{1-L_{f}}\text{ }%
_{H}I_{1^{+}H}^{\alpha }I_{1^{+}}^{1-\gamma }\left\vert \widetilde{u}%
(b)-u(b)\right\vert \\
&=&\left\vert \frac{c_{2}}{c_{1}+c_{2}}\right\vert \frac{(\log t)^{\gamma -1}%
}{\Gamma (\gamma )}\frac{K_{f}}{1-L_{f}}\text{ }_{H}I_{1^{+}}^{\alpha }\text{
}\frac{c_{1}}{c_{2}}\text{\ }_{H}I_{1^{+}}^{1-\gamma }\left\vert u(1)-%
\widetilde{u}(1)\right\vert \\
&=&0.
\end{eqnarray*}%
Thus, $Z_{\widetilde{u}}=Z_{u}.$ From Eq.(\ref{U3}) and Eq.(\ref{U4}), we
have for any $t\in I$%
\begin{align*}
|\widetilde{u}(t)-u(t)|& \leq \Bigg|\widetilde{u}(t)-Z_{\widetilde{u}}-\frac{%
1}{\Gamma (\alpha )}\int_{1}^{t}\big(\log \frac{t}{s}\big)^{\alpha -1}%
\mathcal{F}_{\widetilde{u}}(s)\frac{ds}{s}\Bigg| \\
& \quad +\left\vert Z_{\widetilde{u}}-Z_{u}\right\vert +\frac{1}{\Gamma
(\alpha )}\int_{1}^{t}\big(\log \frac{t}{s}\big)^{\alpha -1}|\mathcal{F}_{%
\widetilde{u}}(s)-\mathcal{F}_{u}(s)|\frac{ds}{s} \\
& \leq \Bigg|\widetilde{u}(t)-Z_{\widetilde{u}}-\frac{1}{\Gamma (\alpha )}%
\int_{1}^{t}\big(\log \frac{t}{s}\big)^{\alpha -1}\mathcal{F}_{\widetilde{u}%
}(s)\frac{ds}{s}\Bigg| \\
& \quad +\frac{K_{f}}{1-L_{f}}\frac{1}{\Gamma (\alpha )}\int_{1}^{t}\big(%
\log \frac{t}{s}\big)^{\alpha -1}\left\vert \widetilde{u}(s)-u(s)\right\vert
ds \\
& \leq B\epsilon +\frac{K_{f}}{\left( 1-L_{f}\right) \Gamma (\alpha )}%
\int_{1}^{t}\big(\log \frac{t}{s}\big)^{\alpha -1}\left\vert \widetilde{u}%
(s)-u(s)\right\vert \frac{ds}{s},
\end{align*}%
and to apply Lemma \ref{le3} with Remark \ref{rem3}$,$\ we obtain 
\begin{equation*}
|\widetilde{u}(t)-u(t)|\leq BE_{\alpha ,1}\left( \frac{K_{f}}{1-L_{f}}(\log
t)^{\alpha }\right) \epsilon ,
\end{equation*}%
Take $C_{f}=BE_{\alpha ,1}(\frac{K_{f}}{1-L_{f}}(\log t)^{\alpha }),$ we get 
\begin{equation*}
|\widetilde{u}(t)-u(t)|\leq C_{f}\epsilon .
\end{equation*}%
Thus, the problem FIDE (\ref{eq1})-(\ref{eq1}) is Ulam-Hyers stable.
\end{proof}

\begin{theorem}
Let the assumptions of Theorem \ref{t1} hold. Then the problem FIDE (\ref%
{eq1})-(\ref{eq1}) is generalized Ulam-Hyers stable.
\end{theorem}

\begin{proof}
In a manner similar to above Theorem \ref{t1}, with putting $\phi
_{f}(\epsilon )=C_{\mathcal{F}}\epsilon $ and $\phi _{f}(1)=0,$ we get $|%
\widetilde{u}(t)-u(t)|\leq \phi _{f}(\epsilon )$.
\end{proof}

\begin{remark}
\label{f1}A function $\widetilde{u}\in C_{1-\gamma ,\log }^{\gamma }[I,E]$
is a solution of the inequality (\ref{16c}) if and only if there exist a
function $h\in C_{1-\gamma ,\log }^{\gamma }[I,E]$ (where $h$ depends on
solution $\widetilde{u}$) such that

\begin{description}
\item[(i)] $|h(t)|\leq \epsilon \varphi (t)$ for all $t\in I$,

\item[(ii)] $_{H}D_{1^{+}}^{\alpha ,\beta }\widetilde{u}(t)=\mathcal{F}_{%
\widetilde{u}}(t)+h(t),\quad t\in I$.
\end{description}
\end{remark}

\begin{theorem}
\label{t2}Assume that the assumptions of Theorem \ref{t3} are satisfied. Let 
$\varphi \in C_{1-\gamma ,\log }[I,E]$ an increasing function and there
exists $\lambda _{\varphi }>0$ such that for any $t\in I$, 
\begin{equation*}
_{H}I_{1^{+}}^{\alpha }\varphi (t)\leq \lambda _{\varphi }\varphi (t).
\end{equation*}%
Then the problem FIDE (\ref{eq1})-(\ref{eq1}) is Ulam-Hyers-Rassias stable.
\end{theorem}

\begin{proof}
Let $\epsilon >0$, and let $\widetilde{u}\in C_{1-\gamma ,\log }^{\gamma
}[I,E]$ be a function, which satisfies the inequality (\ref{16c}), and let $%
u\in C_{1-\gamma ,\log }^{\gamma }[I,E]$ is the unique solution of
Hilfer-Hadamard FIDE (\ref{n1})-(\ref{n2}), that is 
\begin{equation*}
u(t)=Z_{u}+\frac{1}{\Gamma (\alpha )}\int_{1}^{t}\bigg(\log \frac{t}{s}\bigg)%
^{\alpha -1}\mathcal{F}_{u}(s)\frac{ds}{s},
\end{equation*}%
where $Z_{u}$ is defined by Eq.(\ref{w3}). On the other hand, in view of
Remark \ref{f1} with using Lemma \ref{lem5}, and Eq.(\ref{w1}), we get 
\begin{eqnarray}
\widetilde{u}(t) &=&Z_{\widetilde{u}}-\frac{c_{2}}{c_{1}+c_{2}}\frac{(\log
t)^{\gamma -1}}{\Gamma (\gamma )}\frac{1}{\Gamma (1-\gamma +\alpha )}%
\int_{1}^{b}\bigg(\log \frac{b}{s}\bigg)^{\alpha -\gamma }h(s)\frac{ds}{s} 
\notag \\
&&+\frac{1}{\Gamma (\alpha )}\int_{1}^{t}\bigg(\log \frac{t}{s}\bigg)%
^{\alpha -1}\mathcal{F}_{\widetilde{u}}(s)\frac{ds}{s}+\frac{1}{\Gamma
(\alpha )}\int_{1}^{t}\bigg(\log \frac{t}{s}\bigg)^{\alpha -1}h(s)\frac{ds}{s%
}.  \notag
\end{eqnarray}%
By integration of the inequality (\ref{16c}) with Remark \ref{f1}, it
follows that 
\begin{eqnarray*}
&&\Bigg|\widetilde{u}(t)-Z_{\widetilde{u}}-\frac{1}{\Gamma (\alpha )}%
\int_{1}^{t}\bigg(\log \frac{t}{s}\bigg)^{\alpha -1}\mathcal{F}_{\widetilde{u%
}}(s)\frac{ds}{s}\Bigg| \\
&\leq &\left( \left\vert \frac{c_{2}}{c_{1}+c_{2}}\right\vert \frac{(\log
b)^{\gamma -1}}{\Gamma (\gamma )}+1\right) \epsilon \lambda _{\varphi
}\varphi (t).
\end{eqnarray*}
\end{proof}

For sake of brevity, we take $\widetilde{B}=\left( \left\vert \frac{c_{2}}{%
c_{1}+c_{2}}\right\vert \frac{(\log b)^{\gamma -1}}{\Gamma (\gamma )}%
+1\right) $. Consequently, we have

\begin{align*}
|\widetilde{u}(t)-u(t)|& \leq \Bigg|\widetilde{u}(t)-Z_{\widetilde{u}}-\frac{%
1}{\Gamma (\alpha )}\int_{1}^{t}\big(\log \frac{t}{s}\big)^{\alpha -1}%
\mathcal{F}_{\widetilde{u}}(s)\frac{ds}{s}\Bigg| \\
& \quad +\left\vert Z_{\widetilde{u}}-Z_{u}\right\vert +\frac{1}{\Gamma
(\alpha )}\int_{1}^{t}\big(\log \frac{t}{s}\big)^{\alpha -1}|\mathcal{F}_{%
\widetilde{u}}(s)-\mathcal{F}_{u}(s)|\frac{ds}{s} \\
& \leq \Bigg|\widetilde{u}(t)-Z_{\widetilde{u}}-\frac{1}{\Gamma (\alpha )}%
\int_{1}^{t}\big(\log \frac{t}{s}\big)^{\alpha -1}\mathcal{F}_{\widetilde{u}%
}(s)\frac{ds}{s}\Bigg| \\
& \quad +\frac{K_{f}}{1-L_{f}}\frac{1}{\Gamma (\alpha )}\int_{1}^{t}\big(%
\log \frac{t}{s}\big)^{\alpha -1}\left\vert \widetilde{u}(s)-u(s)\right\vert
ds \\
& \leq \widetilde{B}\epsilon \lambda _{\varphi }\varphi (t)+\frac{K_{f}}{%
\left( 1-L_{f}\right) \Gamma (\alpha )}\int_{1}^{t}\big(\log \frac{t}{s}\big)%
^{\alpha -1}\left\vert \widetilde{u}(s)-u(s)\right\vert \frac{ds}{s},
\end{align*}

and to apply Lemma \ref{le3} with Remark \ref{rem3}, we get 
\begin{equation*}
|\widetilde{u}(t)-u(t)|\leq \widetilde{B}\lambda _{\varphi }E_{\alpha
,1}\left( \frac{K_{f}}{1-L_{f}}(\log t)^{\alpha }\right) \epsilon \lambda
_{\varphi }\varphi (t)\ \ t\in \lbrack 1,b].
\end{equation*}%
Take $C_{f,\varphi }=\widetilde{B}\lambda _{\varphi }E_{\alpha ,1}\left( 
\frac{K_{f}}{1-L_{f}}(\log t)^{\alpha }\right) \lambda _{\varphi },$ we can
write%
\begin{equation*}
|\widetilde{u}(t)-u(t)|\leq C_{f,\varphi }\epsilon \varphi (t),\ \ t\in I.
\end{equation*}

This proves that the problem FIDE (\ref{eq1})-(\ref{eq1}) is
Ulam-Hyers-Rassias stable.

\begin{theorem}
Let the assumptions of Theorem \ref{t2} hold. Then the problem FIDE (\ref%
{eq1})-(\ref{eq1}) is generalized Ulam-Hyers-Rassias stable.

\begin{proof}
Set $\epsilon =1$ in the proof of Theorem \ref{t2}, we obtain%
\begin{equation*}
|\widetilde{u}(t)-u(t)|\leq C_{f,\varphi }\varphi (t),\ \ t\in I.
\end{equation*}
\end{proof}
\end{theorem}

\section{An example\label{se5}}

In this section, we consider some particular cases of Hilfer-Hadamard type
of the nonlinear implicit fractional differential equations to apply our
results. We believe that the best way to understand the results obtained
here is through present an example. Then, we use similar ideas to those used
by several researchers in recent studies, involving the existence and
uniqueness of solutions of implicit fractional differential equations,.
Furthermore, for other examples of implicit fractional differential
equations, we suggest \cite{SMS1,AB2,BB,KAR,SKU}. Consider the following
Hilfer-Hadamard FIDE of the form: 
\begin{equation}
_{H}D_{1^{+}}^{\alpha ,\beta }u(t)=\frac{e^{-(\log t)}}{2+e^{\log t}}\Bigg[1+%
\frac{|u(t)|}{1+|u(t)|}+\frac{|_{H}D_{1^{+}}^{\alpha ,\beta }u(t)|}{%
1+|_{H}D_{1^{+}}^{\alpha ,\beta }u(t)|}\Bigg],\quad t\in \lbrack 1,e],
\label{ee1}
\end{equation}%
with the boundary condition%
\begin{equation}
_{H}I_{1^{+}}^{1-\gamma }2u(1)+u(e)=\phi .  \label{ee2}
\end{equation}%
Here, $\alpha =\frac{1}{3},\beta =\frac{2}{3}$, $\gamma =\frac{7}{9}$, $%
c_{1}=2,c_{2}=1,$ and

\begin{equation*}
f(t,u,v)=\frac{e^{-(\log t)}}{2+e^{\log t}}\bigg[1+\frac{u}{1+u}+\frac{v}{1+v%
}\bigg].
\end{equation*}

Clearly, the function $f$ satisfies the hypothesis (H1). Let $E=%
\mathbb{R}
^{+}.$ Then for any $u,v,\bar{u},\bar{v}\in 
\mathbb{R}
^{+}$ and $t\in \lbrack 1,e]$, we find that 
\begin{equation*}
|f(t,u,v)-f(t,\bar{u},\bar{v})|\leq \frac{1}{3}|u-\bar{u}|+\frac{1}{3}|v-%
\bar{v}|.
\end{equation*}%
Hence, the hypothesis (H3) is satisfied with $K_{f}=L_{f}=\frac{1}{3}$. In
addition, the inequality (\ref{eq21}) holds too. Indeed, by some simple
computations, we obtain 
\begin{eqnarray*}
\mathcal{A} &\mathcal{=}&\left( \left\vert \frac{c_{2}}{c_{1}+c_{2}}%
\right\vert \frac{1}{\Gamma (\alpha +1)}+\frac{\mathcal{B}(\gamma ,\alpha )}{%
\Gamma (\alpha )}\right) (\log b)^{\alpha }\frac{K_{f}}{1-L_{f}} \\
&=&\left( \frac{1}{3}\frac{1}{\Gamma (\frac{4}{3})}+\frac{\Gamma (\frac{7}{9}%
)}{\Gamma (\frac{10}{9})}\right) (\log e)^{\frac{1}{3}}\frac{1}{2}\simeq
0.82<1.
\end{eqnarray*}%
Simple computations show that all conditions of Theorem \ref{t3} are
satisfied. It follows that the problem (\ref{ee1})-(\ref{ee2}) has a unique
solution in $C_{\frac{2}{9},\log }^{\frac{7}{9}}([1,e],%
\mathbb{R}
^{+})$.

On the other hand, let $u,v\in 
\mathbb{R}
^{+}$ and $t\in \lbrack 1,e]$, it is easy to see that 
\begin{equation*}
|f(t,u,v)|\leq \frac{e^{-(\log t)}}{2+e^{\log t}}\left( 1+|u|+|v|\right) .
\end{equation*}%
Hence, the hypothesis (H2) is satisfied with $\delta (t)=\sigma (t)=\rho (t)=%
\frac{e^{-(\log t)}}{2+e^{\log t}}$. Moreover, $\delta ^{\ast }=\rho ^{\ast
}=\underset{t\in \lbrack 1,e]}{\sup }\left\vert \frac{e^{-(\log t)}}{%
2+e^{\log t}}\right\vert =\frac{1}{3}<1.$ Thus, the condition (\ref{e1}) is
satisfied. Indeed, 
\begin{eqnarray*}
\Omega &=&\frac{1}{1-\rho ^{\ast }}\left( \left\vert \frac{c_{2}}{c_{1}+c_{2}%
}\right\vert \frac{1}{\Gamma (\gamma )}+1\right) \sigma ^{\ast }\frac{%
\mathcal{B}(\gamma ,\alpha )}{\Gamma (\alpha )}(\log b)^{\alpha } \\
&=&\frac{3}{2}\left( \frac{1}{3}\Gamma (\frac{7}{9})+1\right) \frac{\Gamma (%
\frac{7}{9})\Gamma (\frac{1}{3})}{\Gamma (\frac{1}{3})\Gamma (\frac{10}{9})}%
\frac{1}{3}(\log e)^{\frac{1}{3}}\simeq 0.88<1.
\end{eqnarray*}%
It follows from Theorem \ref{TH3.2}, that problem (\ref{ee1})-(\ref{ee2})
has a solution in $C_{\frac{2}{9},\log }^{\frac{7}{9}}[I,%
\mathbb{R}
^{+}].$

From the above example the Ulam-Hyers and Ulam-Hyers-Rassias stability of
Hilfer-Hadamardtype FIDE with boundary condition are verified by Theorems %
\ref{t1} and \ref{t2}.\newline

\section{Conclusions}

We can conclude that the main results of this article have been successfully
achieved, that is, through of Schaefer fixed point theorem, Banach
contraction principle, Arzela-Ascoli, and some fundamental results in
nonlinear analysis. In the first part of this paper, we established the
equivalence between the Cauchy boundary condition and its mixed type
integral equation through a variety of tools of some properties of
fractional calculus and weighted spaces of continuous functions. In the
second part, we investigated the existence and uniqueness of solutions of
FIDE involving boundary condition and Hilfer-Hadamard fractional derivative.
In the third part, we discussed the stabilities of Ulam-Hyers, generalized
Ulam Hyers, Ulam-Hyers-Rassias, and generalized Ulam-Hyers-Rassias via
generalized Gronwall inequality. In addition, an illustrative example
introduced to justify our results.

There are some articles that carried out a brief study on existence,
uniqueness, and stability of solutions of fractional differential equations
on different types operators, however on Hilfer-Hadamard type operator are
just a few. Also, it should be noted that the results obtained in the
weighted functions space $C_{1-\gamma ,\log }[I,E]$ are contributions to the
growth of the fractional analysis.

\section*{Acknowledgement}

The authors would like to thank the referees for their careful reading of
the manuscript and insightful comments, which helped improve the quality of
the paper. The first author is grateful to the UGC, New Delhi for the award
of National Fellowship for Persons with Disabilities
No.F./2014-15/RGNF-2014-15D-OBC-MAH-84864.\newline

\end{document}